\documentclass[12pt]{article}
\pdfoutput=1
\usepackage{fullpage}
\usepackage{hyperref,amsthm,amssymb,amsmath,verbatim}
\usepackage{bbm}

\usepackage{lmodern}
\usepackage[T1]{fontenc}

\usepackage{cite}

\newcommand{\R}{{\mathbb R}}
\newcommand{\C}{\mathbb{C}}
\newcommand{\U}{\mathbb{U}}
\newcommand{\tr}{\operatorname{tr}}

\newtheorem{lemma}{Lemma}
\newtheorem{theor}{Theorem}
\newtheorem{corol}{Corollary}

\allowdisplaybreaks
\begin{document}

\title{Univariate interpolation by exponential functions and Gaussian RBFs for generic sets of nodes}
\author{D. Yarotsky\footnote{Institute for Information Transmission Problems, Moscow,
\href{mailto:yarotsky@datadvance.net}{\nolinkurl{yarotsky@datadvance.net}}
}
}
\date{October 22, 2012}
\maketitle
\begin{abstract}
We consider interpolation of univariate functions on arbitrary sets of nodes  by Gaussian radial basis functions or by exponential functions. 
We derive closed-form expressions for the interpolation error based on the Harish-Chandra-Itzykson-Zuber formula. 
We then prove the exponential convergence of interpolation for functions analytic in a sufficiently large domain. 
As an application, we prove the global exponential convergence of optimization by expected improvement for such functions.
\end{abstract}

\tableofcontents

\section{Introduction}
In this paper we consider univariate interpolation by Gaussian radial basis functions (RBF) and the closely related interpolation by exponential functions.

RBF interpolation is widely used in applications due to its simplicity and ability to handle generic scattered multidimensional data \cite{buhmann2003radial,
Wendland_2005}.
Our interest in RBF interpolation  is motivated by its role in the optimization by expected improvement, as this interpolation determines the mean of a stationary isotropic Gaussian field after conditioning on a finite number of measurements \cite{Jones:1998:EGO}. 
We restrict our attention to the Gaussian (squared-exponential) RBF, which is one of the most popular examples of analytic RBFs.  
In \cite{yarotsky_ei} we proved that optimization by expected improvement with the corresponding correlation function may be inconsistent for infinitely smooth functions. 
One of the goals of the present paper is to rule out this inconsistency for functions analytic in a sufficiently large domain.
The main ingredient in the proof is a convergence result for RBF interpolation. 

Convergence of RBF interpolation  has been studied extensively in recent years \cite{buhmann2003radial,
Wendland_2005}. General convergence results are most naturally stated for interpolated functions from ``native'' spaces associated with the considered RBF. For a class of multivariate RBFs including the Gaussian, a strong  theorem of this type has been proved by Madych and Nelson \cite{Madych:1992}. This theorem, however, is not sufficient for our purposes, for two reasons. Firstly, the native space for the Gaussian RBF is very narrow (in particular, all functions from this space are entire and vanish at infinity for real values of the arguments). Secondly, and more importantly, this theorem establishes convergence under assumption of an asymptotically dense filling of the design space by the sequence of interpolation nodes. While it is a standard assumption for interpolation when considered as a stand-alone procedure, it may not hold in the context of optimization, where the nodes are determined by the optimization algorithm rather than freely prescribed in advance. For analytic RBFs and interpolated functions, however, one can expect the dense filling to be an excessive requirement due to the non-locality of analytic dependencies.
 
We therefore turn our attention to the univariate case, which is significantly simpler than the multivariate case, and where one can hope to obtain much more complete results, especially for the Gaussian covariance function. Indeed, it is known that in the limit of increasingly flat  rescaled Gaussian RBFs the univariate RBF interpolation is equivalent to the polynomial interpolation \cite{Driscoll2002a}.
In \cite{Platte:2005:PPT:1072891.1084593} Platte and Driscoll have established, by a change of variables, a relation between polynomial interpolation and interpolation by Gaussian RBFs for sets of equally spaced nodes.

There also exists a very simple relation between interpolation by Gaussian RBFs and interpolation by linear combinations of simple exponential functions (sometimes called \emph{exponential ridge functions} in the multivariate setting).  
In the context of multivariate interpolation this relation appeared, in particular, in the work of Schaback \cite{schaback} connecting interpolation by Gaussian RBF to the ``least'' polynomial interpolation of de Boor and Ron. See also the work of Zwicknagl \cite{zwicknagl}, where interpolation by exponential functions is considered as an example of a general class of interpolations based on power series kernels.   

In this paper we establish a further connection between univariate interpolation by polynomials, Gaussian RBFs and exponential functions, by deriving in the latter two cases general error formulas analogous to the well-known error formula for the polynomial interpolation. These formulas are based on the Harish-Chandra-Itzykson-Zuber integral \cite{harish,iz,gr}, which was earlier used by Bos and De Marchi \cite{demarchi} to determine the distribution of nodes maximizing the determinant of the Gaussian RBF interpolation matrix. Our error formulas are valid for arbitrary 1D sets of nodes. They are derived in Section \ref{sec:2}.

In Section \ref{sec:3} we use these formulas to prove convergence of interpolations of analytic functions by exponentials or by Gaussian RBFs for generic infinite sequences of nodes. In particular, this result does not require the nodes to densely fill the design space. 

Finally, in Section \ref{sec:4} we prove the global exponential convergence of optimization by expected improvement for analytic functions as a straightforward application of the interpolation convergence theorem.

\section{Closed-form interpolation error formulas 
} \label{sec:2}

We consider linear interpolation of univariate functions by linear combinations of given basis functions $f_1,f_2,\ldots$. Given a set of distinct nodes $x_1,\ldots,x_n\in\R$ and a function $f$, 
we define the interpolant  $If$ by  
\begin{equation*}
I f=\sum_{k=1}^nc_k f_k,
\end{equation*}
where the coefficients $c_k$ are chosen so that  
\begin{equation*}
I f(x_l)=f(x_l),\quad l=1,\ldots,n.
\end{equation*}
We will occasionally write the operator $I$ as $I_{\{x_k\}_{k=1}^n}$ or $I_n$ to emphasize the dependence on the nodes or their number. 

A particular type of interpolation is specified by the choice of basis functions $f_k$. We will consider the following types:
\begin{itemize}
\item \emph{Interpolation by Gaussian RBF} (denoted $I^{\rm g}$ or $I^{\rm g}_{\{x_k\}_{k=1}^n}$)
corresponds to Gaussians centered at the interpolation nodes $x_k$: 
\begin{equation*}
f_k(x) = e^{-(x-x_k)^2/2}.
\end{equation*}
\item For any distinct values $t_1,\ldots, t_n\in \R$,  \emph{interpolation by exponential functions} (denoted $I^{\rm e}$ or $I^{\rm e}_{\{x_k, t_k\}_{k=1}^n}$)  corresponds to  
\begin{equation*}
f_k(x) = e^{t_k x}.
\end{equation*}
\item \emph{Polynomial interpolation} ($I^{\rm p}$ or $I^{\rm p}_{\{x_k\}_{k=1}^n}$) corresponds to \begin{equation*}
f_k(x)=x^{k-1}.
\end{equation*}
\end{itemize}

For the interpolation $I$ to be well-defined, the matrix $(f_k(x_l))_{k,l=1}^n$ must be nondegenerate. This is so for the above three types:  for $I^{\rm g}$ this follows e.g. from the positive definiteness of the function $e^{-x^2/2}$; for $I^{\rm p}$ this follows from the nondegeneracy of the Vandermonde matrix; for $I^{\rm e}$ this follows e.g. from the arguments below. 

Gaussian interpolation $I_{\{x_k\}_{k=1}^n}^{\rm g}$ reduces to exponential interpolation $I_{\{x_k, x_k\}_{k=1}^n}^{\rm e}$ (i.e., with $t_k\equiv x_k$) by noting that 
\begin{equation*}
e^{-(x-x_k)^2/2} =e^{-x^2/2} e^{x_k x} e^{-x_k^2/2}.
\end{equation*}
Indeed, thanks to this identity we can write 
\begin{equation*}
I_{\{x_k\}_{k=1}^n}^\mathrm{g} f(x) = \sum_{k=1}^n c_k e^{-(x-x_k)^2/2} 
= e^{-x^2/2}\sum_{k=1}^n (c_k e^{-x_k^2/2}) e^{x_k x} 
= e^{-x^2/2} I_{\{x_k, x_k\}_{k=1}^n}^\mathrm{e} \widetilde{f}(x),
\end{equation*}
where  $\widetilde{f}(x)=e^{x^2/2}f(x)$. In other words, the interpolation operators are related by the identity
\begin{equation}\label{eq:g2e}
I_{\{x_k\}_{k=1}^n}^\mathrm{g} = e^{-\hat{x}^2/2} \circ I_{\{x_k, x_k\}_{k=1}^n}^\mathrm{e} \circ e^{\hat{x}^2/2},
\end{equation}
where $e^{\pm \hat{x}^2/2}$ is the operator of multiplication by the function $e^{\pm {x}^2/2}$.

This argument shows in particular that the interpolation $I_{\{x_k, t_k\}_{k=1}^n}^\mathrm{e}$ is well-defined, i.e. its interpolation matrix 
\begin{equation*}
A = (e^{t_kx_m})_{k,m=1}^n
\end{equation*} 
is invertible, at least if $t_k\equiv x_k$. 
In fact,  this interpolation is well-defined for any sets of distinct values $t_1,\ldots,t_n$ and distinct nodes $x_1,\ldots,x_n$. 
One way to see this is to use the remarkable formula of Harish-Chandra-Itzykson-Zuber (HCIZ).
To introduce this formula, we need a few definitions.
Consider the diagonal matrices    
\begin{equation*}
X=\operatorname{diag}(x_1,\ldots,x_n), \quad T=\operatorname{diag}(t_1,\ldots,t_n).
\end{equation*}
Let $V(X)$ denote the Vandermonde determinant for the points $x_1,\ldots,x_n$:
\begin{equation*}
V(X)=\det(x_k^m)_{\substack{1\le k\le n\\ 0\le m\le n-1}}=\prod_{1\le k< l\le n}(x_l-x_k).
\end{equation*}
Finally, define the constant $\beta_n$ by 
\begin{equation*}
\beta_n=\prod_{k=0}^{n-1}k!.
\end{equation*}
Then the HCIZ formula reads \cite{harish,iz,gr}:
\begin{equation}\label{eq:hciz}
\det A= \beta_n^{-1}V(X)V(T)\int_{\U(n)}e^{\tr(TU^{\dagger}XU)}dU,
\end{equation}
where integration is over Haar measure on the group $\U(n)$ of unitary matrices of size $n$. 
Here and in the sequel by $^\dagger$ we denote the Hermitian conjugate. 

Note that the integrand in \eqref{eq:hciz} is strictly positive. 
Since $V(X)\ne 0$ and $V(T)\ne 0$ for distinct $t_1,\ldots,t_n$ and  $x_1,\ldots,x_n$, it follows in particular that $\det A\ne 0$, as stated above.

The main results of this section are HCIZ-integral-based error formulas for the interpolations $I^\mathrm{g}$ and $I^\mathrm{e}$. 

We first consider the $I^\mathrm{e}$ case. We introduce some additional notation:
\begin{itemize}
\item Let $Z_{\{x_k, t_k\}_{k=1}^n}=\int_{\U(n)}e^{\tr(TU^{\dagger}XU)}dU>0$ be the integral appearing in the HCIZ formula \eqref{eq:hciz}.
\item Let $ \int_{S^{2n+1}} \cdot \,d\mathbf{v}$ denote integration over the normalized Lebesgue measure on the unit sphere $S^{2n+1}=\{\mathbf{v}: |\mathbf{v}|=1\}$ in $\C^{n+1}$.
\item Let $\widetilde{X}$ be the extension of the diagonal matrix $X$ by the value $x$:
\begin{equation*}
\widetilde{X}=\operatorname{diag}(x, x_1,\ldots,x_n).
\end{equation*}
\item Let $\operatorname{conv}(\widetilde{X})\subset\R$ be the convex hull of points $x,x_1,\ldots,x_n$.
\item Let $P_\mathbf{v}:\C^n\to \C^{n+1}$ be any isometry between $\C^n$ and the orthogonal complement to the vector $\mathbf{v}\in S^{2n+1}$ in $\C^{n+1}$; $P_\mathbf{v}$ is assumed to depend measurably on $\mathbf{v}$. 
\end{itemize}
\begin{theor}\label{thm:1} For any $f\in C^n(\operatorname{conv}(\widetilde{X}))$,
\begin{align}\label{eq:ieerror}
f(x)-I_{\{x_k, t_k\}_{k=1}^n}^\mathrm{e} f(x) = & \frac{ \prod_{k=1}^n (x-x_k) }{ n! Z_{\{x_k, t_k\}_{k=1}^n}}
\nonumber\\
&\times \int_{S^{2n+1}} \int_{\U(n)}
e^{\tr(T U^{\dagger} P_\mathbf{v}^{\dagger} \widetilde{X} P_\mathbf{v} U)} 
\Big[\prod_{k=1}^n \big(\frac{d}{dq}-t_k\big)\Big] 
f(q)\Big|_{q = \mathbf{v}^{\dagger} \widetilde{X}\mathbf{v}} d\mathbf{v}dU.
\end{align}
\end{theor}

\begin{proof}

We first prove formula \eqref{eq:ieerror} for functions of the form $f(x)=e^{tx}$ with any $t\in\C$, and then extend it to all $f\in C^n(\operatorname{conv}(\widetilde{X}))$.

It will be convenient in the following to consider $t$ and $x$ as elements extending the sequences  $t_1,\ldots, t_n$ and $x_1,\ldots,x_n$, respectively, by identifying
\begin{equation*}
t_0 = t,\quad x_0=x.
\end{equation*}

We begin by recalling the following classical result from the general theory of linear interpolation:

\begin{lemma}[see e.g. Theorem 3.8.1 in \cite{davis}]
Let $I$ be any linear interpolation with distinct nodes $x_1,\ldots,x_n$ and basis functions $f_1,\ldots,f_n$.
Then, assuming $\det (f_k(x_m))_{k,m=1}^n\ne 0$, the error of interpolation of a function $f_0$ is given by
\begin{equation*}
f_0(x_0)-If_0(x_0) = \frac{\det(f_k(x_m))_{k,m=0}^n}{\det(f_k(x_m))_{k,m=1}^n}.
\end{equation*}
\end{lemma}
As a consequence, 
\begin{equation*}
f(x)-I_{\{x_k, t_k\}_{k=1}^n}^{\rm e}f(x) = \frac{\det(e^{t_k x_m})_{k,m=0}^n}{\det(e^{t_k x_m})_{k,m=1}^n}.
\end{equation*}
We apply the HCIZ formula to both numerator and denominator and obtain
\begin{equation*}
f(x)-I_{\{x_k, t_k\}_{k=1}^n}^\mathrm{e} f(x) = \frac{ \prod_{k=1}^n [(x-x_k)(t-t_k)] }{ n! Z_{\{x_k, t_k\}_{k=1}^n}}
\int_{\U(n+1)}
e^{\tr(\widetilde{T} \widetilde{U}^{\dagger}  \widetilde{X} \widetilde{U})} 
d\widetilde{U},
\end{equation*}
where integration is performed over unitary matrices of size $n+1$, and 
\begin{equation*}
\widetilde{X}=\operatorname{diag}(x, x_1,\ldots,x_n), \quad \widetilde{T}=\operatorname{diag}(t, t_1,\ldots,t_n).
\end{equation*}
Let us write the $(n+1)$-dimensional trace in this formula as the sum of the part corresponding to the first entry of the matrix $\widetilde{T}$ and the remaining $n$-dimensional trace.
To this end, denote by $\mathbf{v}$ the first column of the matrix $\widetilde{U}$, and by $\widetilde{U}'$ denote the remaining $(n+1)\times n$ sub-matrix. 
Then we can write
\begin{equation*}
\tr(\widetilde{T} \widetilde{U}^{\dagger}  \widetilde{X} \widetilde{U}) = 
\tr({T} \widetilde{U}'^{\dagger}  \widetilde{X} \widetilde{U}')+t\mathbf{v}^{\dagger} \widetilde{X}\mathbf{v}. 
\end{equation*}
We can replace integration over $\widetilde{U}$ by double integration, first over $\mathbf{v}\in S^{2n+1}$ and then over the complementary matrices $\widetilde{U}'$. It is convenient to fix for each given $\mathbf{v}$ one complementary matrix $P_{\mathbf{v}}$, and then make the substitution   
\begin{equation*}
\widetilde{U}' = P_{\mathbf{v}}U,
\end{equation*}
where $U\in \U(n)$. In this way we reduce integration over $\U(n+1)$ to integration over $S^{2n+1} \times \U(n)$. 
The resulting measure of integration is the product of the normalized Lebesgue measure on the sphere with Haar measure on $\U(n)$. 
As a result, we get 
\begin{align*}
f(x)-I_{\{x_k, t_k\}_{k=1}^n}^\mathrm{e} f(x)  = & \frac{ \prod_{k=1}^n (x-x_k) }{ n! Z_{\{x_k, t_k\}_{k=1}^n} }\\
&\times \int_{S^{2n+1}} \int_{\U(n)}
e^{\tr(T U^{\dagger} P_\mathbf{v}^{\dagger} \widetilde{X} P_\mathbf{v} U)} 
e^{t\mathbf{v}^{\dagger} \widetilde{X}\mathbf{v}} 
\prod_{k=1}^n (t-t_k)
d\mathbf{v}dU
\end{align*}
Since 
\begin{equation*}
\Big[\prod_{k=1}^n \big(\frac{d}{dq}-t_k\big)\Big] 
e^{tq}\Big|_{q = \mathbf{v}^{\dagger} \widetilde{X}\mathbf{v}}
=
e^{t\mathbf{v}^{\dagger} \widetilde{X}\mathbf{v}} 
\prod_{k=1}^n (t-t_k),
\end{equation*}
the proof is complete for $f(x)=e^{tx}$.

It remains to extend formula \eqref{eq:ieerror} to all functions $f\in C^n(\operatorname{conv}(\widetilde{X}))$.
This can be done by standard arguments, using the formula's linearity.
First note that the formula holds for all polynomials, by multiply differentiating it written for $f(x)=e^{tx}$   with respect to $t$ at $t=0$. 
Then, for any $f\in C^n(\operatorname{conv}(\widetilde{X}))$, apply the Weierstrass theorem to $f^{(n)}$ to show that for any $\epsilon$ there is a polynomial $p$ such that $|f^{(k)}(x')-p^{(k)}(x')|<\epsilon$ for all $x'\in \operatorname{conv}(\widetilde{X})$ and all derivatives $k=0,1,\ldots,n$. 
\end{proof}

\noindent
{\bf Remark.}  Formula \eqref{eq:ieerror} leaves some freedom for the choice of $P_{\mathbf{v}}$. One natural choice is the one diagonalizing the matrix  
$P_\mathbf{v}^{\dagger} \widetilde{X} P_\mathbf{v}$ and placing its eigenvalues according to the order of eigenvalues in $X$; see the proof of Theorem \ref{thm:2}  in the next section.

\medskip
It follows from \eqref{eq:g2e} that the interpolation errors for $I^\mathrm{g}$ and $I^\mathrm{e}$ are simply related by  
\begin{equation}\label{eq:errg2e}
\mathbbm{1}-I_{\{x_k\}_{k=1}^n}^\mathrm{g} = 
e^{-\hat{x}^2/2}\circ (\mathbbm{1}- I_{\{x_k,x_k\}_{k=1}^n}^\mathrm{e})\circ e^{\hat{x}^2/2}.
\end{equation}
Theorem 1 then immediately implies an error formula for the Gaussian interpolation:
\begin{corol} For any $f\in C^n(\operatorname{conv}(\widetilde{X}))$,
\begin{align*} 
f(x)-& I_{\{x_k\}_{k=1}^n}^\mathrm{g} f(x) =  \frac{ \prod_{k=1}^n (x-x_k) }{ n! Z_{\{x_k,x_k\}_{k=1}^n} }\\
&\times \int_{S^{2n+1}} \int_{\U(n)}
e^{\tr(X U^{\dagger} P_\mathbf{v}^{\dagger} \widetilde{X} P_\mathbf{v} U)} 
e^{-x^2/2}\Big[\prod_{k=1}^n \big(\frac{d}{dq}-x_k\big)\Big] 
e^{q^2/2}f(q)\Big|_{q = \mathbf{v}^{\dagger} \widetilde{X}\mathbf{v}} d\mathbf{v}dU.
\end{align*}
\end{corol}

Theorem \ref{thm:1} also allows us to show that the exponential interpolation converges to the polynomial one if $t_k\to 0 $ for all $k$, by viewing error formula \eqref{eq:ieerror} as a generalization of the classical error formula for polynomial interpolation.

\begin{corol}\label{corol:2} For any $f\in C^n(\operatorname{conv}(\widetilde{X}))$,
\begin{equation*}
\lim_{\{t_k\to 0\}_{k=1}^n} I_{\{x_k,t_k\}_{k=1}^n}^{\rm e}f(x)=I_{\{x_k\}_{k=1}^n}^{\rm p}f(x).
\end{equation*}
\end{corol}
\begin{proof}
Recall the well-known error expression for the polynomial interpolation based on the Hermite-Genocchi formula for divided differences:
\begin{equation}\label{eq:iperror}
f(x)-I^\mathrm{p}_{\{x_k\}_{k=1}^n} f(x) =  \frac{ \prod_{k=1}^n (x-x_k) }{ n!}
\int_{\Delta_n} f^{(n)}\big(\sum_{k=0}^n s_k x_k\big) d\mathbf{s},
\end{equation}
where $\mathbf s = (s_0, s_1, \ldots s_n)$ and the integration is over normalized Lebesgue measure on  the $n$-dimensional simplex $\Delta_n =\{\mathbf s: \sum_{k=0}^n s_k = 1, s_k\ge 0\}$.

To prove the corollary, we use error formulas \eqref{eq:ieerror} and \eqref{eq:iperror} to show that 
\begin{equation*}
\lim_{\{t_k\to 0\}_{k=1}^n} \big(f(x)- I_{\{x_k,t_k\}_{k=1}^n}^{\rm e}f(x)\big)=f(x)-I_{\{x_k\}_{k=1}^n}^{\rm p}f(x).
\end{equation*}

Indeed, first observe that in this limit the differential operator on the r.h.s. of \eqref{eq:ieerror} tends to $d^n/dq^n$, $Z_{\{x_k, t_k\}_{k=1}^n}$ tends to 1, and the exponential factor in the integrand tends to 1 so that the dependence on $U$ vanishes making  integration over $\U(n)$ trivial: 

\begin{equation*}
\lim_{\{t_k\to 0\}_{k=1}^n}\big(f(x)-I_{\{x_k,t_k\}_{k=1}^n}^\mathrm{e} f(x)\big) =  \frac{ \prod_{k=1}^n (x-x_k) }{ n!}
\int_{S^{2n+1}}
f^{(n)}
(\mathbf{v}^{\dagger} \widetilde{X}\mathbf{v}) d\mathbf{v}.
\end{equation*}

To see how the integration over $S^{2n+1}$ transforms, write $\mathbf{v} = (a_0+ib_0, \ldots, a_n+ib_n)$ and substitute $a_k=\sqrt{s_k}\cos\phi_k, b_k = \sqrt{s_k}\sin \phi_k$ for each $k$; the Jacobian of this substitution equals $2^n$. Then, using Dirac's delta $\delta$ and the identity $\delta(|\mathbf v|-1)=2\delta((|\mathbf v|-1)(|\mathbf v|+1))= 2\delta(|\mathbf v|^2-1)$,

\begin{align*}
\int_{S^{2n+1}} &
f^{(n)}
(\mathbf{v}^{\dagger} \widetilde{X}\mathbf{v}) d\mathbf{v} \\
& =
\frac{1}{\operatorname{Vol}_{2n+1}(S^{2n+1})} \int_{\R^{2n+2}} f^{(n)} 
(\mathbf{v}^{\dagger} \widetilde{X}\mathbf{v}) \delta(|\mathbf v|-1) d\mathbf{v}\\
& =
\frac{2}{\operatorname{Vol}_{2n+1}(S^{2n+1})} \int_{\R^{2n+2}} f^{(n)} 
(\mathbf{v}^{\dagger} \widetilde{X}\mathbf{v}) \delta(|\mathbf v|^2-1) d\mathbf{v}\\
& =
\frac{n!}{(2\pi)^{n+1}} \int_{\R^{2n+2}} f^{(n)} \Big(\sum_{k=0}^n(a_k^2+b_k^2)x_k\Big) \delta\Big( \sum_{k=0}^n 
(a_k+b_k)^2 -1\Big) \prod_{k=0}^n da_k db_k \\
& = \frac{n!}{(2\pi)^{n+1}} 
\int_{\left\{\substack{0\le s_k\\ 0\le \phi_k\le 2\pi}\right\}_{k=0}^n} f^{(n)} \Big(\sum_{k=0}^n s_k x_k\Big) \delta\big(\sum\nolimits_{k=0}^n s_k-1\big)\prod_{k=0}^n ds_k d\phi_k \\
& = n! 
\int_{\{0\le s_k\}_{k=0}^n} f^{(n)} \Big(\sum_{k=0}^n s_k x_k\Big) \delta\big(\sum\nolimits_{k=0}^n s_k-1\big)\prod_{k=0}^n ds_k \\
&= \int_{\Delta_n} f^{(n)}\big(\sum_{k=0}^n s_k x_k\big) d\mathbf{s}. \end{align*}
\end{proof}

\section{Convergence of interpolation for analytic functions}\label{sec:3}
In this section we use Theorem \ref{thm:1} to prove convergence of interpolation on a bounded segment $[a,b]\subset\R$ for functions $f$ analytic in a sufficiently large complex domain $\mathcal D\subset \C$.
We state our convergence theorem simultaneously for all three types of interpolation appearing in the previous section, thus emphasizing the similarity between them. 

\begin{theor} \label{thm:2}
Suppose $f$ is analytic in a complex domain $\mathcal D\supset [a,b]$, and  $\operatorname{dist}([a,b],\partial\mathcal D)>\rho>0$.
Let $I$ denote any of the interpolations $I^{\mathrm{e}}, I^{\mathrm{g}}, I^{\mathrm{p}}$ for a sequence of distinct nodes $x_1,x_2,\ldots\subset[a,b]$;
in the case of $I^{\mathrm{e}}$ assume additionally that there exists $R$ such that $|t_k|<R$ for all $k$.

Then 
\begin{equation}\label{thm:main}
\sup_{x\in[a,b]} |f(x) - I_n f(x)|\le c \big(\frac{b-a}{\rho}\big)^n
\end{equation}
with some constant $c=c(f,a,b,\rho,R)$.
\end{theor}
\begin{proof}
We give the proof only for $I^{\mathrm{e}}$;  the result for $I^{\mathrm{g}}$ then follows immediately from relation \eqref{eq:errg2e}, while the result for the polynomial interpolation can be seen as a trivial special case thanks to Corollary \ref{corol:2} of Theorem \ref{thm:1}.

It is convenient to represent the error $f(x)-I^{\mathrm{e}}_n f(x)$ in the form
\begin{equation*}
f(x)-I^{\mathrm{e}}_n f(x) = \frac{
\int_{S^{2n+1}} \int_{\U(n)}\widetilde{K}_n(x,\mathbf{v}, U) \phi_n(x,\mathbf{v}) d\mathbf{v} dU
}{
\int_{\U(n)} {K_n (U)} dU
}, 
\end{equation*}
where 

\begin{align}
\phi_n(x,\mathbf{v}) &= \frac{ \prod_{k=1}^n (x-x_k) }{ n!}
\Big[\prod_{k=1}^n \big(\frac{d}{dq}-t_k\big)\Big] 
f(q)\Big|_{q = \mathbf{v}^{\dagger} \widetilde{X}_n\mathbf{v}},\label{eq:phin}\\
\widetilde{K}_n(x,\mathbf{v}, U) &= e^{\tr(T_n U^{\dagger} P_\mathbf{v}^{\dagger} \widetilde{X}_n P_\mathbf{v} U)},\nonumber\\
K_n(U) &=  e^{\tr(T_n U^{\dagger} X_nU)}.\nonumber
\end{align}

\begin{lemma} \label{lm:phi} Under assumptions of Theorem \ref{thm:2}, 
\begin{equation*}
\sup_{x\in[a,b],\mathbf{v}\in S^{2n+1}}|\phi_n(x,\mathbf{v})| \le c_1 \big(\frac{b-a}{\rho}\big)^n
\end{equation*}
with some constant $c_1=c_1(f,a,b,\rho,R)$.
\end{lemma}
\begin{proof}
Choose a contour $\gamma\subset \mathcal D$ enclosing the segment $[a,b]$ so that 
\begin{equation*}
\min_{z\in \gamma, q\in[a,b]} |z-q|>\rho.
\end{equation*}
We use Cauchy's formula for $f(q)$,
\begin{equation*}
f(q) = \frac{1}{2\pi i}\oint_{\gamma}\frac{f(z)dz}{z-q},
\end{equation*}
and substitute it in \eqref{eq:phin}.
Expanding the product of first order differential operators,
\begin{equation*}
\Big|\frac{1}{n!}\Big[\prod_{k=1}^n \big(\frac{d}{dq}-t_k\big)\Big] \frac{1}{z-q}\Big| 
\le\frac{1}{n!}\sum_{s=0}^n \binom{n}{s}\frac{(n-s)!R^s}{|z-q|^{n-s+1}}
\le\frac{1}{\rho^{n+1}}\sum_{s=0}^n\frac{\rho^s R^s}{s!} \le\frac{e^{\rho R}}{\rho^{n+1}}.
\end{equation*}
Since $|x-x_k|\le b-a$ for all $k$, it follows that 
\begin{equation*}
|\phi_n(x,\mathbf{v})|\le (b-a)^n \frac{e^{\rho R}}{\rho^{n+1}} \frac{1}{2\pi }\oint_{\gamma}|f(z)| |dz|,
\end{equation*}
which implies lemma's claim with
\begin{equation*}
c_1 = \frac{e^{\rho R}}{2\pi\rho }\oint_{\gamma}|f(z)| |dz|.
\end{equation*}
\end{proof}

At this point we need to specify the choice of $P_{\mathbf{v}}$. 
Let $\sigma$ be a permutation of $x_1,\ldots,x_n$ in the increasing order:
\begin{equation*}
x_{\sigma(1)}\le\ldots\le x_{\sigma(n)}.
\end{equation*}
Choose $P_{\mathbf{v}}$ so that  $P_\mathbf{v}^{\dagger} \widetilde{X}_n P_\mathbf{v}
= \operatorname{diag} (\widetilde{x}_1,\ldots,\widetilde{x}_n)$, where 
\begin{equation*}
\widetilde{x}_{\sigma(1)}\le\ldots\le \widetilde{x}_{\sigma(n)}.
\end{equation*}

\begin{lemma}\label{lm:kk}
With the above choice of $P_{\mathbf{v}}$,
\begin{equation}\label{eq:lm2}
e^{R(a-b)}\le\frac{\widetilde{K}_n(x,\mathbf{v}, U)}{K_n(U)}\le e^{R(b-a)}.
\end{equation}
for any $x\in[a,b], \mathbf{v}\in S^{2n+1}$ and $U\in\U(n)$.
\end{lemma}
\begin{proof}
We have 
\begin{align*}
\Big|\ln\frac{\widetilde{K}_n(x,\mathbf{v}, U)}{K_n(U)}\Big| &=
|\tr(UT_n U^{\dagger} 
(P_\mathbf{v}^{\dagger} \widetilde{X}_n P_\mathbf{v}-X_n))|\\
&\le \|T_n\|\operatorname{tr}|P_\mathbf{v}^{\dagger} \widetilde{X}_n P_\mathbf{v}-X_n|\\
&\le R \operatorname{tr}|P_\mathbf{v}^{\dagger} \widetilde{X}_n P_\mathbf{v}-X_n|,
\end{align*}
where  we have used the well-known  inequality $\operatorname{tr}(BC)\le \|B\|\operatorname{tr}|C|$ with $|C|=(C^\dagger C)^{1/2}$.

Thanks to our choice of $P_{\mathbf{v}}$, $P_\mathbf{v}^{\dagger} \widetilde{X}_n P_\mathbf{v}-X_n$ is a diagonal operator, and
\begin{equation*}
\operatorname{tr}|P_\mathbf{v}^{\dagger} \widetilde{X}_n P_\mathbf{v}-X_n|
= \sum_{k=1}^n |\widetilde{x}_{\sigma(k)}-x_{\sigma(k)}|.
\end{equation*}  
It remains to observe that 
\begin{equation}\label{eq:eig}
\sum_{k=1}^n |\widetilde{x}_{\sigma(k)}-x_{\sigma(k)}|\le b-a.
\end{equation}
Recall that the eigenvalues of a restriction of a quadratic form to a subspace of co-dimension 1 alternate with the original eigenvalues; 
in particular the eigenvalues of $P_\mathbf{v}^{\dagger} \widetilde{X}_n P_\mathbf{v}$ alternate with the eigenvalues of $ \widetilde{X}_n $.
From this and from our convention on the order of eigenvalues it is easy to see that the open intervals $\{(\widetilde{x}_{\sigma(k)}, x_{\sigma(k)})\}_{k=1}^n$ do not overlap. 
For example, if $x\in[\min_{k=1,\ldots,n} x_k, \max_{k=1,\ldots,n} x_k]$, then there exists $n_0$ such that $x_{\sigma(n_0)}\le x\le x_{\sigma(n_0)+1}$, and we can write
\begin{equation*}
x_{\sigma(1)} \le \widetilde{x}_{\sigma(1)}\le\ldots \le
x_{\sigma(n_0)}\le \widetilde{x}_{\sigma(n_0)}\le
x\le \widetilde{x}_{\sigma(n_0)+1} \le x_{\sigma(n_0)+1}\le \ldots \le
\widetilde{x}_{\sigma(n)}\le {x}_{\sigma(n)},
\end{equation*}
which makes the absence of overlapping clear. 
The other cases, $x<\min_{k=1,\ldots,n} x_k$ and $x>\max_{k=1,\ldots,n} x_k$, are considered similarly.
Since all the intervals $\{(\widetilde{x}_{\sigma(k)}, x_{\sigma(k)})\}_{k=1}^n$ at the same time lie in $[a,b]$, we conclude \eqref{eq:eig}. 
\end{proof}

The claim \eqref{thm:main} of the theorem now follows immediately from  Lemma \ref{lm:phi} and the upper bound in \eqref{eq:lm2}, with
\begin{equation*}
c = e^{ R(b-a)}c_1 = \frac{e^{(\rho+b-a) R}}{2\pi\rho }\oint_{\gamma}|f(z)| |dz|.
\end{equation*}
\end{proof}
The above theorem proves convergence  only if the analyticity domain of the interpolated function is sufficiently large. It is known that if the domain is not large enough, the interpolants may diverge: in the case of polynomial interpolation this is the well-known Runge phenomenon \cite{runge}, and a similar effect holds for the RBF interpolation, see \cite{Platte:2005:PPT:1072891.1084593,
fornberg}.

\section{Optimization by expected improvement for analytic functions}\label{sec:4}
In this section we describe an application of Theorem \ref{thm:2} to optimization by expected improvement (EI).
Optimization by EI is a kind of stochastic Bayesian optimization popular in engineering application \cite{Jones:1998:EGO}. 
We consider the simplest version of the algorithm with a centered Gaussian process and a fixed covariance function.

Suppose that we are searching for the global minimum of a function $f$ on a segment $[a,b]$. We iteratively sample points $x_1,x_2,\ldots \subset [a,b]$, and evaluate the function $f$ at these points. For each $n$ we define the current best result as 
\begin{equation*}
f_n^* = \min_{k=1,\ldots,n} f(x_k).
\end{equation*} 
The question is whether $f_n^*$ converges to the global minimum 
\begin{equation*}
f^*=\min_{x\in[a,b]}f(x),
\end{equation*} and how fast if yes.

In optimization by EI $f$ is assumed to be a realization of a centered Gaussian process $\{\xi_x\}_{x\in[a,b]}$ with a given covariance function $G(x,x')=\mathsf E(\xi_{x}\xi_{x'})$, 
and the choice of each $x_{n+1}$ is determined from the history $\{x_k,f(x_k)\}_{k=1}^{n}$ by maximizing the expectation of improvement of the current best result for the process conditioned on the event $\{\xi_{x_k}=f(x_k)\}_{k=1}^{n}$:
\begin{equation*}
x_{n+1} = \arg \max_{x\in [a,b]} \mathfrak{I}_n (x),
\end{equation*}
where
\begin{equation*}
\mathfrak{I}_n (x) 
= \mathsf{E} \big( f^*_{n} - \min (f^*_n, \xi_x)
\big| \{\xi_{x_k} = f(x_k)\}_{k=1}^n\big).
\end{equation*}
In this way, each optimization iteration is reduced to an auxiliary optimization problem $\mathfrak{I}_n\to\max_x,$ which can be written in an analytic form and readily solved numerically for moderate values of $n$.   
See \cite{yarotsky_ei} for more details and a bibliography on EI.

In the sequel we assume that the auxiliary optimization problem is exactly solved at each step $n$.

A popular choice of  covariance function in optimization by EI is the Gaussian function  
\begin{equation}\label{eq:gcov}
G(x,x')=G(x-x') = e^{-(x-x')^2/2}.
\end{equation} 
In \cite{yarotsky_ei}, we have proved that the optimization by EI with this covariance function does not in general converge to the global optimum for $C^{\infty}$ functions $f$. 
We prove now that if $f$ is analytic in a sufficiently large complex neighborhood of $[a,b]$, then the optimization does converge, and moreover with an exponential convergence rate.    
\begin{theor}
Consider optimization of a (real-valued) function $f$  on the segment $[a,b]$ by EI with covariance function \eqref{eq:gcov}.
Suppose that $f$ continues analytically to a complex domain $\mathcal D\supset  [a,b]$ such that $\operatorname{dist}([a,b],\partial\mathcal D)>\rho>|b-a|$. 
Then $f_n^*$ converges to the global minimum $f^*$ of $f$ on $[a,b]$, and 
\begin{equation*}
f^*_n-f^* = O\Big(\big(\frac{b-a}{\rho}\big)^n\Big),\quad n\to\infty.
\end{equation*}
\end{theor}
\begin{proof}
Let $m_n(x)$ and $\sigma^2_n(x)$ denote the posterior mean and variance of the process $\xi_x$ conditioned on the event $\{\xi_{x_k}=f(x_k)\}_{k=1}^{n}$:
\begin{equation*}
\xi_x|\{\xi_{x_k} = f(x_k)\}_{k=1}^n\sim\mathcal N(
m_n(x), \sigma^2_n(x)).
\end{equation*}
A straightforward computation with Gaussians shows that $m_n(x)$ is the interpolation of $f$ by the  RBF  $G$ with the nodes $x_1,\ldots,x_n$, i.e.
\begin{equation*}
m_n(x) = I^{\mathrm{g}}_nf(x).
\end{equation*}
We then obtain from Theorem \ref{thm:2} that for some constant $c$ 
\begin{equation}\label{eq:fm}
\max_{x\in[a,b]} |f(x)-m_n(x)|\le c \big(\frac{b-a}{\rho}\big)^n.
\end{equation}
On the other hand, it follows from results proved in \cite{yarotsky_ei} (see Theorem 2 and the example immediately below) that in the case of  covariance \eqref{eq:gcov} the posterior variance $\sigma_n^2(x)$ converges faster than exponentially to 0 uniformly on $[a,b]$,  for any sequence $x_1,x_2,\ldots \subset [a,b]$:
\begin{equation}\label{eq:sigma}
\max_{x\in[a,b]} \sigma_n^2(x) = O(\epsilon^n), \quad n\to\infty,
\end{equation} 
for any $\epsilon>0$.

We will prove the theorem by showing that for sufficiently large $n$
\begin{equation}\label{eq:fn}
f^*_{n+1}-f^* \le 3c \big(\frac{b-a}{\rho}\big)^n,
\end{equation}
where $c$ is from \eqref{eq:fm}.

Fix $n$. Since $f^*_{n+1}\le f^*_{n}$, it suffices to prove \eqref{eq:fn} in the case when 
\begin{equation}\label{eq:fnf}
f^*_{n}-f^*>3c \big(\frac{b-a}{\rho}\big)^n.
\end{equation}
Suppose that \eqref{eq:fn} does not hold. 
Then $f(x_{n+1})-f^* > 3c \big(\frac{b-a}{\rho}\big)^n$ and, by \eqref{eq:fm}, 
\begin{equation}\label{eq:m1}
m_n(x_{n+1})-f^* > 2c \big(\frac{b-a}{\rho}\big)^n.
\end{equation}
Consider now the minimizer $x^*\in[a,b]$ for which $f(x^*) = f^*$. 
Again using \eqref{eq:fm}, we have
\begin{equation}\label{eq:m2}
m_n(x^*)-f^* < c \big(\frac{b-a}{\rho}\big)^n.
\end{equation} 
We see by comparing \eqref{eq:m1} with \eqref{eq:m2} that the expected value of the function $f$ at $x^*$ is lower than at $x_{n+1}$. 
By exploiting the smallness of the variance, expressed by \eqref{eq:sigma}, 
we will conclude that  $x^*$ provides a better expected improvement than $x_{n+1}$, and thus will reach a contradiction with the definition of $x_{n+1}$. 
\begin{lemma}
For all $n$ and $x\in[a,b]$
\begin{equation}\label{eq:lm3}
|\mathfrak{I}_n(x)-(f_n^*-\min(f_n^*,m_n(x)))|\le \sigma_n(x).
\end{equation}
\end{lemma} 

\begin{proof} Let $\chi(s) = f_n^*-\min(f_n^*,s)$. 
Then we can write the l.h.s. of \eqref{eq:lm3} as 
\begin{equation*}
|\mathsf{E}(\chi(\xi_{x;n})-\chi(m_n(x)))|,
\end{equation*} 
where by $\xi_{x;n}$ we denote the conditioned process: $\xi_{x;n} = \xi_x|\{\xi_{x_k} = f(x_k)\}_{k=1}^n$.
But since $|\chi(s_1)-\chi(s_2)|\le |s_1-s_2|$ for all $s_1,s_2$, we have 
\begin{equation*}
|\mathsf{E}(\chi(\xi_{x;n})-\chi(m_n(x)))|\le 
\mathsf{E}|\xi_{x;n}-m_n(x)|\le
\sqrt{\mathsf{E}(\xi_{x;n}-m_n(x))^2}=
\sigma_n(x).
\end{equation*}
\end{proof}
Applying this lemma to $x=x_{n+1}$ and $x=x^*$, we get
\begin{align*}
\mathfrak{I}_n(x_{n+1})&\le f_n^*-\min(f_n^*,m_n(x_{n+1})) +\sigma_n(x_{n+1}),\\
\mathfrak{I}_n(x^*)&\ge f_n^*-\min(f_n^*,m_n(x^*)) -\sigma_n(x^*),
\end{align*} 
which implies
\begin{equation*}
\mathfrak{I}_n(x^*)-\mathfrak{I}_n(x_{n+1})\ge 
[\min(f_n^*,m_n(x_{n+1}))-\min(f_n^*,m_n(x^*))] -\sigma_n(x_{n+1})- \sigma_n(x^*).
\end{equation*}
Inequalities \eqref{eq:fnf}, \eqref{eq:m1} and \eqref{eq:m2} imply that
the expression in brackets here is greater than $c \big(\frac{b-a}{\rho}\big)^n$.
Also, thanks to \eqref{eq:sigma},  $\sigma_n(x)<\frac{c}{3} \big(\frac{b-a}{\rho}\big)^n$ 
for all all sufficiently large $n$ and all $x\in[a,b]$, in particular for $x_{n+1}$ and $x^*$. 
We conclude that
\begin{equation*}
\mathfrak{I}_n(x^*)-\mathfrak{I}_n(x_{n+1})\ge \frac{c}{3} \big(\frac{b-a}{\rho}\big)^n>0,
\end{equation*}
which completes the proof. 
\end{proof}

We end this section with a brief discussion of the obtained result. 
Note that it is of course a consequence of the strong assumption of analyticity that the convergence is both \emph{global} and \emph{exponential} (compare with the local exponential convergence of classical gradient-based numerical optimization and with the global power law convergence of EI optimization for finitely smooth functions \cite{bull}).

Note also that the strong claim of this theorem only pertains to the convergence $f^*_n\to f^*$; we have not at all claimed the convergence $x_n\to x^*$ or $f(x_n)\to f^*$.

Finally, we remark that one must be careful with practical implementations of the EI algorithm when used with covariance function \eqref{eq:gcov} and applied to analytic objective functions, as the algorithm involves ill-conditioned interpolation matrices and other elements potentially sensitive to round-off errors and/or requiring high-precision computations (see \cite{yarotsky_ei}).

\section*{Acknowledgement}
\addcontentsline{toc}{section}{Acknowledgement}
The author thanks the anonymous referees for several valuable suggestions and corrections.

\end{document}